\documentclass[11pt]{amsart}

\usepackage{amsmath}
\usepackage{amsthm}
\usepackage{hyperref}
\usepackage[margin=1.0in]{geometry}

\numberwithin{equation}{section}

\newtheorem{prop}{Proposition}
\newtheorem{lemma}[prop]{Lemma}

\newtheorem{thm}[prop]{Theorem}
\newtheorem{cor}[prop]{Corollary}

\numberwithin{prop}{section}

\theoremstyle{definition}
\newtheorem{defn}[prop]{Definition}

\newtheorem{rmk}[prop]{Remark}

\newcommand{\del}{\partial}
\newcommand{\delb}{\bar{\partial}}\newcommand{\dt}{\frac{\partial}{\partial t}}
\newcommand{\brs}[1]{\left| #1 \right|}

\newcommand{\gG}{\Gamma}

\newcommand{\gD}{\Delta}

\newcommand{\gs}{\sigma}

\newcommand{\gU}{\Upsilon}
\newcommand{\gL}{\Lambda}

\newcommand{\gl}{\lambda}

\newcommand{\gw}{\omega}
\newcommand{\ga}{\alpha}
\newcommand{\gb}{\beta}

\newcommand{\N}{\nabla}
\newcommand{\FF}{\mathcal F}

\renewcommand{\bar}[1]{\overline{#1}}

\renewcommand{\i}{\sqrt{-1}}

\newcommand{\bga}{\bar{\alpha}}

\newcommand{\bj}{\bar{j}}
\newcommand{\bk}{\bar{k}}
\newcommand{\bl}{\bar{l}}
\newcommand{\bm}{\bar{m}}
\newcommand{\bn}{\bar{n}}

\newcommand{\IP}[1]{\left<#1\right>}

\DeclareMathOperator{\Sym}{Sym}
\DeclareMathOperator{\Rc}{Rc}

\DeclareMathOperator{\tr}{tr}

\DeclareMathOperator{\ev}{ev}

\begin{document}

\title[Pluriclosed flow on manifolds with globally generated
bundles]{Pluriclosed flow on manifolds with globally generated bundles}

\begin{abstract} We show global existence and convergence results for the
pluriclosed flow on manifolds for which certain naturally associated tensor
bundles are globally generated.
\end{abstract}

\author{Jeffrey Streets}
\address{Rowland Hall\\
         University of California, Irvine\\
         Irvine, CA 92617}
\email{\href{mailto:jstreets@uci.edu}{jstreets@uci.edu}}

\thanks{J. Streets gratefully acknowledges support from the NSF via DMS-1341836,
DMS-1454854 and from the
Alfred P. Sloan Foundation via a Sloan Research Fellowship.}

\date{June 12, 2016}

\maketitle

\section{Introduction}

Given $(M^{2n}, J)$ a complex manifold, we say that a Hermitian metric $g$ is
pluriclosed if the associated K\"ahler form $\gw$ satisfies $\i \del\delb \gw =
0$.  For such metrics the author and Tian introduced \cite{ST2} a parabolic flow
generalizing the K\"ahler-Ricci flow (see \S \ref{pcfsec} for definitions). 
Recently in \cite{SBIPCF} the author obtained global existence and convergence
results for this flow and manifolds admitting special background metrics, for instance
tori and manifolds with nonpositive bisectional curvature.  In this short note
we establish global existence and convergence results for this flow assuming
conditions of a complex geometric nature as opposed to the differential
geometric assumptions of metrics with certain curvature conditions.  Thus these
theorems are more natural from a complex geometry standpoint, and apply to a much wider class of manifolds.  Moreover, our
results have implications for the existence and moduli of generalized K\"ahler
structures on these manifolds using the generalized K\"ahler-Ricci flow \cite{STGK}.  This
note 
is a close companion to \cite{SBIPCF}, and though we will review the most
pertinent aspects, familiarty with that paper will help in reading this.  Before
stating our theorems we record several definitions.

\begin{defn} Fix $(M^{2n}, J)$ a complex manifold.  Given $g$ a Hermitian metric
on $M$, by taking inverses and tensor products $g$ defines a Hermitian metric on
$(T_{1,0})^{\otimes p} \oplus (T^*_{1,0})^{\otimes q}$.  Then by restriction we
obtain a natural metric on any subbundle $E \subset (T_{1,0})^{\otimes p} \oplus
(T^{*}_{1,0})^{\otimes q}$, which we will refer to as $F_E(g)$.  We say that
such a holomorphic subbundle $E $ is 
\begin{enumerate}
\item \emph{covariant proper} if $E \subset (T_{1,0})^{\otimes p}$ for some $p \in \mathbb N$ and the natural map
\begin{align*}
F_E : \Sym^2(T^*_{1,0}) \to \Sym^2(E^*)
\end{align*}
is proper.
\item \emph{covariant weakly proper} if $E \subset (T_{1,0})^{\otimes p}$ for some $p \in \mathbb N$, and if given a background metric $h$, the map
\begin{align*}
F_E : \Sym^2 (T^*_{1,0}) \cap \left\{ g\ |\ \frac{\det g}{\det h} \geq 1
\right\} \to \Sym^2(E^*)
\end{align*}
is proper.
\item \emph{contravariant proper} if $E \subset( T^*_{1,0})^{\otimes p}$ for some $p \in \mathbb N$ and the natural map
\begin{align*}
F_E : \Sym^2(T_{1,0}) \to \Sym^2(E^*)
\end{align*}
is proper.
\item \emph{globally generated} if it is generated by sections.  That is,
letting $H^0(M, E)$ denote the finite dimensional space of holomorphic sections of $E$,
for all $p \in M$ the natural evaluation map
\begin{align*}
\ev_p : H^0(M, E) \to E_p
\end{align*}
is surjective.
\end{enumerate}
\end{defn}

\begin{rmk} 
\begin{enumerate}
\item An elementary interpretation of a bundle being proper is that an upper
bound for the if the induced metric on the bundle implies an upper bound for the
original metric.  For weakly proper bundles the meaning is that an upper bound
for the induced metric on the bundle combined with a lower bound for the
determinant implies an upper bound for the metric.
\item The most basic examples of proper bundles are $T_{1,0}, T_{1,0}^*$.  Other
examples include $(T_{1,0})^{\otimes p}, (T^*_{1,0})^{\otimes p}$.
\item An example of a bundle which is weakly proper but not proper is
$\Lambda^{n-1}(T_{1,0})$.
\item The question of when complex manifolds admit globally generated bundles has various relations to algebraic geometry.  We direct the reader to \cite{Bryant} and the references therein for some further context.
\end{enumerate}
\end{rmk}

Next we state our main theorems.  They are stated in an overly general manner, but we supplement the discussion with concrete families of examples.

\begin{thm} \label{contramplethm} Let $(M^{2n}, J)$ be a compact complex
manifold.
\begin{enumerate}
\item Suppose $M$ admits a contravariant globally generated proper bundle. 
Given $g$ a pluriclosed metric the solution to pluriclosed flow exists on
$[0,\tau^*)$ (see Definition \ref{taustardef} for the definition of $\tau^*$).
\item Suppose $M$ admits a contravariant globally generated weakly proper
bundle.  If $c_1^{BC} = 0$ and $[\del \gw] = 0 \in H^{2,1}$ then the solution
exists on $[0,\infty)$ and converges exponentially as $t \to \infty$ to a
Calabi-Yau metric.
\end{enumerate}
\end{thm}

\begin{rmk}
\begin{enumerate}
\item The bundle $T^*_{1,0}$ is proper.  If it is globally generated then the
anticanonical bundle is also globally generated and it follows that the formal
existence time $\tau^* = \infty$.
\item K\"ahler manifolds with globally generated cotangent bundle are quite abundant.  For instance, any product of Riemann surfaces of positive
genus yields a manifold with globally generated $T^*_{1,0}$.  Moreover,
having a globally generated cotangent bundle is inherited by complex subvarieties, so in particular subvarieties of tori have globally generated cotangent bundles.  This includes large families of manifolds of general type.
\item The hypothesis $[\del \gw] = 0$ is satisfied automatically in some circumstances, such as of course if $h^{2,1} = 0$, or if $(M^{2n}, J)$ satisfies the $\del\delb$-lemma.
\end{enumerate}
\end{rmk}

\begin{thm} \label{covamplethm} Let $(M^{2n}, J)$ be a complex manifold with a
covariant globally generated weakly proper bundle and $c_1^{BC} = 0$.  Given $g$
a pluriclosed metric with $[\del \gw] = 0 \in H^{2,1}$ the solution to
pluriclosed flow with initial condition $g$ exists on $[0,\infty)$ and converges
exponentially as $t \to \infty$ to a Calabi-Yau metric.
\end{thm}

\begin{rmk}
\begin{enumerate}
\item The two cohomological hypotheses $c_1^{BC} = 0$ and $[\del \gw] = 0 \in
H^{2,1}$ are natural to impose if one expects convergence to Calabi-Yau.  The
Hopf surface $S^3 \times S^1$ with standard complex structure has $c_1 = 0$, but
$c_1^{BC} \neq 0$ and $[\del \gw] \neq 0$ for any pluriclosed metric.
\item The bundle $T^{1,0}$ is proper, and this bundle being globally generated is equivalent
to $(M^{2n}, J)$ being complex homogeneous, as follows from elementary arguments
(see \cite{Akhiezer}).  Theorem \ref{covamplethm} can be used to rule out the
existence of pluriclosed metrics on certain backgrounds as well.  For instance,
compact quotients of $SL(2,\mathbb C)$ are parallelizable and hence $c_1^{BC} =
0$.  Moreover, they satisfy $h^{2,1} = 0$ (\cite{Akhiezer} Corollary 8.2.3).  It
follows that these manifolds admit no pluriclosed metric, since Theorem
\ref{covamplethm} then yields a K\"ahler metric, which quotients of
$SL(2,\mathbb C)$ cannot support since the only K\"ahler parallelizable
manifolds are tori.  This particular statement can be obtained directly by
averaging a putative SKT metric and then performing direct calculations using
the Lie algebra structure of $SL(2,\mathbb C)$.  Nonetheless we include this
example to illustrate the nonexistence principle.
\end{enumerate}
\end{rmk}

\begin{cor} \label{GKcor} Let $(M^{2n}, I)$ be a complex manifold with either a
covariant or contravariant globally generated weakly proper bundle and $c_1^{BC}
= 0$.  Suppose $(M^{2n}, I, J, g)$ is a generalized K\"ahler structure with
$[\del \gw_I] = 0 \in H^{2,1}_I$.  Then $(M^{2n}, I, J, g)$ is deformable
through generalized K\"ahler structures to a structure $(M^{2n}, I,
J_{\infty},g_{\infty})$ such that $g_{\infty}$ is Calabi-Yau.
\end{cor}

The central new observation in proving the above results is that there are very
clean evolution equations for the square norm of a holomorphic section of a
vector bundle along pluriclosed flow.  The hypotheses on the bundle allow one to
turn these favorable evolution equations into upper or lower bounds for the
metric.  Combining these with prior results yields full regularity of the flow.
 We make use of the Perelman functionals for pluriclosed flow discovered in
\cite{ST3} to obtain the convergence statements.  In \S \ref{bckgrnd} we provide
relevant background information on the pluriclosed flow and generalized
K\"ahler-Ricci flow.  In \S \ref{secev} we derive evolution equations for
holomorphic sections of tensor bundles over $M$.  We combine our estimates in \S
\ref{proofsec} and give the proofs of the theorems.

\vskip 0.1in

\textbf{Acknowledgements}:  The author would like to thank Ben McKay, Robert
Bryant, and Gueo Grantcharov for helpful comments.  This article was prepared for a special issue of the journal  ``Complex geometry and Lie Groups'' associated to the conference of the same name held in Nara, Japan March 22nd-26th.  The author thanks Anna Fino, Ryushi Goto, and Keizo Hasegawa for the kind invitation.

\section{Background} \label{bckgrnd}

In this section we give a very brief introduction to relevant aspects of
pluriclosed flow and generalized K\"ahler-Ricci flow.  The reader should refer
to \cite{SBIPCF}, \cite{ST2}, and \cite{ST3} for more detail.

\subsection{Pluriclosed flow} \label{pcfsec}

In this subsection we record some elementary properties of the pluriclosed flow.
First we express the flow
equation using differential operators
appearing in Hodge theory.  In particular, on a complex manifold $(M^{2n}, J)$,
a one-parameter family of Hermitian metrics $g_t$ is a solution of pluriclosed
flow if the corresponding K\"ahler forms $\gw_t$ satisfy
\begin{align} \label{pcfH}
\dt \gw = \del \del^*_{g} \gw + \delb \delb^*_{g} \gw +
\i \del\delb \log \det g.
\end{align}
As shown in \cite{ST2}, this is a strictly parabolic equation with pluriclosed
initial condition $\gw_0$, and admits short-time solutions on compact manifolds.

It is also useful to express this flow using the Chern connection.  Given
$(M^{2n}, J, g)$ a Hermitian manifold, the Chern connection is the unique
connection $\N$ on $T_{1,0}$ such that $\N g \equiv 0$, $\N J \equiv 0$ and
the torsion of $\N$ has vanishing $(1,1)$ piece.  This torsion can be expressed in complex coordinates
as
\begin{align*}
T_{ij \bk} = g_{l \bk} \left[ \gG_{i j}^l - \gG_{ji}^l \right] = g_{j \bk,i} -
g_{i \bk,j}.
\end{align*}
The metric is K\"ahler if and only if $T \equiv 0$.  Due to the fact that $\N$,
in general, has torsion, there are various ``Ricci
curvatures" which can be defined using this connection.  First, one has
\begin{align*}
S_{i \bj} =&\ g^{\bl k} \Omega_{k \bl i \bj},
\end{align*}
where $\Omega$ is the Chern curvature.  We will also use the
representative of the first Chern class with respect to the Chern connection, which we will denote by
\begin{align*}
 \rho_{i\bj} =&\ g^{\bl k} \Omega_{i\bj k \bl}.
\end{align*}
We also define a certain
quadratic expression in torsion, namely
\begin{align*}
Q_{i \bj} =&\ g^{\bl k} g^{\bn m} T_{i k \bn} T_{\bj \bl m}.
\end{align*}
With these definitions made,
we can express the pluriclosed flow equation (\cite{ST2} Proposition 3.3) as
\begin{align} \label{PCFCC}
\dt g =&\ - S + Q.
\end{align}

\subsection{Formal existence time}

Important in understanding the existence time of solutions to (\ref{pcfH})
(equivalently \ref{PCFCC}) is a formal cohomological obstruction.  Observe that
a pluriclosed metric defines a positive class in Aeppli cohomology.  Using
(\ref{pcfH}), it is direct to see that this class evolves along the pluriclosed
flow via
\begin{align*}
[\gw_t] = [\gw_0] - t c_1.
\end{align*}
This allows us to define a formal maximal smooth existence time (cf. \cite{ST3})

\begin{defn} \label{taustardef} Given $(M^{2n}, J)$ a compact complex manifold,
and $g_0$ a pluriclosed metric, let
\begin{align*}
\tau^* := \sup \{t > 0\ |\ [\gw_0] - t c_1 \mbox{ admits pluriclosed metrics}
\}.
\end{align*}
\end{defn}

For times $\tau < \tau^*$ we can define a reduction of the pluriclosed flow to a
flow on a $(1,0)$ form on $[0,\tau]$.  First, fix a background Hermitian metric
$h$.  Since $\tau < \tau^*$, there exists $\mu
\in
\Lambda^{1,0}$ such that
\begin{align} \label{omegahatdef}
\hat{\gw}_{\tau} := \gw_{0} - \tau \rho(h) + \delb \mu + \del \bar{\mu} > 0.
\end{align}
Now consider the smooth one-parameter family of K\"ahler forms
\begin{align*}
\hat{\gw}_t := \frac{t}{\tau} \hat{\gw}_{\tau} + \frac{(\tau-t)}{\tau} \gw_0.
\end{align*}

\begin{defn} Let $(M^{2n}, g_t, J)$ be a smooth solution to pluriclosed flow on
$[0,\tau]$.  Given choices $\hat{g}_t, h, \mu$ as above, for a one parameter
family $\ga_t \in \Lambda^{1,0}$ let
\begin{align*}
\gw_{\ga} := \hat{\gw}_t + \delb \ga_t + \del \bar{\ga_t}.
\end{align*}
We say that a one-parameter family $\ga_t \in \Lambda^{1,0}$ is a solution to
\emph{$(\hat{g}_t,h,\mu)$-reduced pluriclosed flow} if
\begin{gather} \label{reduced}
\begin{split}
\dt \ga =&\ \delb^*_{g_{\ga}} \gw_{\ga} - \frac{\i}{2} \del \log
\frac{\det g_{\ga}}{\det h} - \frac{\mu}{\tau},\\
\ga_0 =&\ 0.
\end{split}
\end{gather}
\end{defn}

\begin{rmk} \label{reductionrmk} This reduction generalizes the reduction of
K\"ahler-Ricci flow to the complex parabolic Monge Ampere equation (with
additional background terms).  In the special case which we frequently consider where
$c_1^{BC} = 0$ and $[\del \gw_0] = 0$, it follows that $\tau^* = \infty$, and
moreover if one chooses the background metric $h$ to satisfy $\rho(h) = 0$, then
the reduction can be chosen so that $\hat{\gw}_t = \gw_0$, $\mu = 0$.  This is
relevant to obtaining certain a priori estimates below.
\end{rmk}

\subsection{Generalized K\"ahler-Ricci flow} \label{gksec}

We will briefly summarize the generalized K\"ahler-Ricci flow here, referring
the reader to \cite{SBIPCF, STGK} for further detail.   To begin we introduce generalized K\"ahler structures, referring the reader to \cite{Gualtieri} for further background.  A generalized K\"ahler
manifold is a quadruple $(M^{2n}, I,J,g)$ consisting of a Riemannian metric with
two compatible integrable complex structures $I$, $J$, such that the
corresponding K\"ahler forms satisfy
\begin{align*}
d^c_I \gw_I = H = - d^c_J \gw_J, \qquad d H = 0.
\end{align*}
In particular, the metric $g$ is pluriclosed with respect to two different
complex structures.  This observation combined with the connection between
pluriclosed flow and renormalization group flows from \cite{ST3} leads one to the
definition of generalized K\"ahler-Ricci flow:
\begin{align*}
\dt g =&\ - 2 \Rc + \frac{1}{2} H^2 \qquad \dt H = \gD_d H,\\
\dt I =&\ L_{\theta_I^{\sharp}} I, \qquad \dt I = L_{\theta_J^{\sharp}} J.
\end{align*}
Here $\theta_I$ is the Lee form with respect to the Hermitian structure $(g,I)$.
Interestingly, both complex structures evolve by time-dependent, but distinct,
diffeomorphisms.  It is possible to gauge-modify this flow to freeze one complex
structure, but not both.  Choosing to freeze $I$, the resulting metric evolves
by pluriclosed flow on $(M, I)$, and so our results on pluriclosed flow will
have immediate applications to this flow.  We will refer to this as considering the flow ``in the $I$-fixed gauge."

\subsection{Evolution equations and technical results}

We begin by recording two evolution equations relevant to what follows.

\begin{lemma} \label{volumeformev} (\cite{SBIPCF} Lemma 6.1) Let $(M^{2n}, J,
g_t)$ be a solution to
pluriclosed flow, and let $h$ denote another Hermitian metric on $(M, J)$.  Then
\begin{align*}
 \left( \frac{\del}{\del t} - \gD \right) \log \frac{\det g}{\det h} =&\
\brs{T}^2 - \tr_g \rho(h).
\end{align*}
\end{lemma}

\begin{lemma} \label{delalphaev} (\cite{SBIPCF} Proposition 4.9,4.10)  Let
$(M^{2n}, J, g_t)$ be a solution to
pluriclosed flow.  Fix background data $\hat{g}_t, h, \mu$ and a solution
$\ga_t$ to (\ref{reduced}).  Then
\begin{align*}
\dt \brs{\del \ga}^2_{g_{\ga}} =&\ \gD_{g_{\ga}} \brs{\del \ga}^2 - \brs{\N \del
\ga}^2 - \brs{\bar{\N} \del \ga}^2 - 2 \IP{Q, \tr\del \ga \otimes
\delb \bar{\ga}} - 2 \Re \IP{\tr_{g_{\ga}} \N^{g_{\ga}} T_{\hat{g}} + \del \mu,
\delb \bga}.
\end{align*}
Suppose furthermore that $\mu = 0$ and
\begin{align*}
\del \hat{\gw}_t = \del \hat{\gw}_0 = \delb \eta.
\end{align*}
Let $\phi = \del \ga - \eta$.  Then
\begin{align*}
\left(\dt - \gD_{g_t} \right) \brs{\phi}^2 =&\ - \brs{\N \phi}^2 -
\brs{T_{g_t}}^2
- 2 \IP{Q, \phi \otimes \bar{\phi}}.
\end{align*}
\end{lemma}

Next we record some background theorems on regularity and the existence and
rigidity of limit points for pluriclosed flow relevant to what follows.
Corollary \ref{convergencecor} summarizes the situation and is the main technical tool.

\begin{thm} \label{uplowbndthm} (\cite{SBIPCF} Theorem 1.8) Let $(M^{2n}, J)$ be
a compact complex manifold.
 Suppose $g_t$ is a solution to the pluriclosed flow on $[0,\tau)$, with
$\ga_t$ a solution to the $(\hat{g}_t,h,\mu)$-reduced flow. Assume there is a
constant
$\gl$ such that for all $t \in [0,\tau)$,
\begin{align*}
\gl g_0 \leq g_t.
\end{align*}
There exists a constant $\gL = \gL(n,g_0,\hat{g}, h, \mu,\gl)$ such that for all
$t \in
[0,\tau)$,
\begin{align*}
g_t \leq \gL (1 + t) g_0, \qquad \brs{\del \ga}^2 \leq \gL.
\end{align*}
\end{thm}

\begin{thm} \label{EKthm} (\cite{SBIPCF} Theorem 1.7) Let $(M^{2n}, J)$ be a
compact complex manifold.
 Suppose $g_t$ is a solution to the pluriclosed flow on $[0,\tau)$, $\tau \leq
1$, with
$\ga_t$ a solution to the $(\hat{g}_t,h,\mu)$-reduced flow as in
(\ref{reduced}).  Suppose there
exist constants $\gl,\gL$ such that
\begin{align} \label{Ekthmhyp}
\gl g_0 \leq g_t \leq \gL g_0, \qquad \brs{\del \ga}^2 \leq \gL.
\end{align}
Given $k \in \mathbb N$ there exists a constant $C =
C(n,k,g_0,\hat{g},h,\mu,\gl,\gL)$ such that
\begin{align*}
\sup_{M \times \{t\}} t \sum_{j=0}^k \brs{  \N_g^{j} \gU(g,h)}^{\frac{2}{1+j}}
\leq C,
\end{align*}
where $\gU(g,h) = \N^g - \N^h$ is the difference of the Chern connections
associated to $g$ and $h$.
\end{thm}

\begin{lemma} \label{riglemma} (\cite{SBIPCF} Lemma 6.3) Let $(M^{2n}, J, h)$ be
a compact Hermitian
manifold
with $\rho(h) \leq 0$.  Suppose $g$ is a pluriclosed metric which is a steady
gradient
soliton.  Then $g$ is a
Calabi-Yau metric.
\end{lemma}

\begin{cor} \label{convergencecor} Let $(M^{2n}, J, h)$ be a compact Hermitian
manifold
with $\rho(h) \leq 0$.  Suppose $g_t$ is a solution to pluriclosed flow on
$[0,\infty)$ satisfying
\begin{align*}
C^{-1} h \leq g \leq C h, \qquad \brs{\N_h^k \gU(g,h)}^2 \leq C,
\end{align*}
where $\gU(g,h) = \N^g - \N^h$ is the difference of the Chern connections
associated to $g$ and $h$.  Then $g_t$ converges exponentially to a Calabi-Yau
metric.
\begin{proof} This argument is implicit in the proof of (\cite{SBIPCF} Theorem
1.1), though not stated explicitly and so we repeat it for convenience.  With
the assumed uniform estimates, any sequence of times $t_j \to \infty$ admits a
smooth subsequential
limiting metric on the same complex manifold.  Moreover, the assumed uniform
estimates imply that the Perelman-type $\FF$ functional for the
pluriclosed flow (\cite{ST3} Theorem 1.1) has a uniform upper bound for all
times.  It follows from a standard argument that any subsequential limit as
described above is a
pluriclosed steady soliton, and hence by Lemma \ref{riglemma} Calabi-Yau.  It
now follows
from the linear/dynamic stability result of (\cite{ST1} Theorem 1.2) that the
whole flow converges exponentially to $g_{\infty}$, as required.
\end{proof}
\end{cor}

\section{Evolution of holomorphic sections} \label{secev}

\begin{lemma} (\cite{SPCFSTB} Lemma 4.7) \label{01formgenev} Let $(M^{2n},
 J,g_t)$ be a solution to
pluriclosed flow, and
suppose $\gb_t, \mu_t \in (T^*_{1,0})^{\otimes p}$ are one-parameter families satisfying
\begin{align} \label{pformev}
 \dt \gb =&\ {\gD}_{g_t} \gb + \mu.
\end{align}
Then
\begin{align*} 
 \dt \brs{\gb}^2 =&\ \gD \brs{\gb}^2 - \brs{\N \gb}^2 - \brs{\bar{\N} \gb}^2 -
p \IP{Q, \tr_g \left(\gb \otimes \bar{\gb} \right)} + 2 \Re \IP{\gb,\mu}.
\end{align*}
\end{lemma}

\begin{rmk} Lemma 4.7 of \cite{SPCFSTB} is stated only for $\gb \in
\Lambda^{p,0}$ but the proof easily applies to this more general case.
\end{rmk}

\begin{cor} \label{pformevcor} Let $(M^{2n}, J,g_t)$ be a solution to
pluriclosed flow, and
suppose $\gb \in (T^*_{1,0})^{\otimes p}$ is holomorphic.  Then
\begin{align} \label{ptensnormev}
 \dt \brs{\gb}^2 =&\ \gD \brs{\gb}^2 - \brs{\N \gb}^2 -
p \IP{Q, \tr_g \left(\gb \otimes \bar{\gb} \right)}.
\end{align}
\begin{proof} If $\gb$ is holomorphic then $\gb_t = \gb$ is a solution of
(\ref{Aev})
with $\mu = 0$, and so from Lemma \ref{01formgenev} we conclude
\begin{align*}
 \dt \brs{\gb}^2 =&\ \gD \brs{\gb}^2 - \brs{\N \gb}^2 - \brs{\bar{\N} \gb}^2 -
p \IP{Q, \tr_g \left(\gb \otimes \bar{\gb} \right)}\\
=&\ \gD \brs{\gb}^2 - \brs{\N \gb}^2 -
p \IP{Q, \tr_g \left(\gb \otimes \bar{\gb} \right)},
\end{align*}
as required.
\end{proof}
\end{cor}

\begin{lemma} \label{01tensgenev} Let $(M^{2n}, J, g_t)$ be a solution to
pluriclosed flow, and
suppose $A_t , B_t \in T^{p,0}$ are one-parameter families satisfying
\begin{align} \label{Aev}
 \dt A =&\ {\gD}_{g_t} A + B.
\end{align}
Then
\begin{align*}
 \dt \brs{A}^2 =&\ \gD \brs{A}^2 - \brs{\N A}^2 - \brs{\bar{\N} A}^2 +
p \IP{Q, \tr_g \left(A \otimes \bar{A} \right)} + 2 \Re \IP{A,B}.
\end{align*}
\begin{proof} By direct computation we have
 \begin{align*}
  \dt \brs{A}^2 =&\ \dt \left( g_{i_1 \bj_1} \dots g_{i_p \bj_p} {A}^{i_1\dots
i_p} \bar{A}^{\bj_1\dots \bj_p} \right)\\
  =&\ p (-S + Q)_{i_1 \bj_1} g_{i_2 \bj_2} \dots g_{i_p \bj_p} {A}^{i_1 \dots
i_p}
\bar{A}^{\bj_1 \dots \bj_p} + \IP{ \gD A, \bar{A}} +
\IP{A, \bar{\gD} \bar{A}} + 2 \Re \IP{A,B}\\
=&\ p \IP{- S + Q, \tr_g \left(A \otimes \bar{A} \right)} + \IP{ \gD A,
\bar{A}} +
\IP{A, \bar{\gD} \bar{A}} + 2 \Re \IP{A,B}.
 \end{align*}
Next we observe the commutation formula
\begin{align*}
 \bar{\gD} \bar{A}^{\bj_1 \dots \bj_p} =&\ g^{\bk l} \N_{\bk} \N_l
\bar{A}^{\bj_1\dots\bj_p}\\
=&\ g^{\bk l} \N_l
\N_{\bk} \bar{A}_{\bj_1\dots\bj_p} + \sum_{r=1}^p g^{\bk l} \Omega_{\bk l
\bm}^{\bj_r} \bar{A}^{\bj_1\dots \bj_{r-1} \bm \bj_{r+1}\dots\bj_p}\\
=&\ \gD \bar{A}_{\bj_1\dots\bj_p} + \sum_{r=1}^p S^{\bj_r}_{\bm}
\bar{A}^{\bj_1\dots
\bj_{r-1} \bm \bj_{r+1}\dots\bj_p}.
\end{align*}
It follows that
\begin{align*}
 \IP{{A}, \bar{\gD} \bar{A}} =&\ g_{i_1 \bj_1} \dots g_{i_p \bj_p}
{A}^{i_1\dots i_p} \bar{\gD} \bar{A}^{\bj_1 \dots \bj_p}\\
 =&\ g_{i_1 \bj_1} \dots g_{i_p \bj_p} {A}^{i_1\dots i_p} \left[ \gD
\bar{A}^{\bj_1 \dots \bj_p} + \sum_{r=1}^p S^{\bj_r}_{\bm} A^{\bj_1\dots
\bj_{r-1}
\bm \bj_{r+1}\dots\bj_p} \right]\\
 =&\ \IP{ {A}, \gD \bar{A}} + p \IP{S, \tr_g A \otimes \bar{A}}.
\end{align*}
Lastly observe the identity
\begin{align*}
 \gD \brs{A}^2 =&\ \IP{\gD A, \bar{A}} + \IP{{A}, \gD \bar{A}} +
\brs{\N
A}^2 + \brs{\bar{\N} A}^2.
\end{align*}
Combining the above calculations yields the lemma.
\end{proof}
\end{lemma}

\begin{cor} \label{ptensorev} Let $(M^{2n}, J,g_t)$ be a solution to
pluriclosed flow, and
suppose $A \in T^{p,0}$ is holomorphic.  Then
\begin{gather} \label{ptensev}
\begin{split}
  \dt \brs{A}^2 =&\ \gD \brs{A}^2 - \brs{\N A}^2 +
p \IP{Q, \tr_g \left(A \otimes \bar{A} \right)},\\
\dt \log \brs{A}^2 \leq&\ \gD \log \brs{A}^2 + p \brs{T}^2.
\end{split}
\end{gather}
\begin{proof} If $A$ is holomorphic then $A_t = A$ is a solution of (\ref{Aev})
with $B = 0$, and so from Lemma \ref{01formgenev} we conclude
\begin{align*}
 \dt \brs{A}^2 =&\ \gD \brs{A}^2 - \brs{\N A}^2 - \brs{\bar{\N} A}^2 +
p \IP{Q, \tr_g \left(A \otimes \bar{A} \right)}\\
=&\ \gD \brs{A}^2 - \brs{\N A}^2 +
p \IP{Q, \tr_g \left(A \otimes \bar{A} \right)},
\end{align*}
as required.  A further elementary calculation yields that
\begin{align*}
\dt \log \brs{A}^2 =&\ \gD \log \brs{A}^2 + \frac{\brs{\N
\brs{A}^2}^2}{\brs{A}^2} - \brs{\N A}^2 + p \frac{\IP{Q, \tr_g (A \otimes
\bar{A})}}{\brs{A}^2}.
\end{align*}
Since $A$ is holomorphic it follows from Kato's inequality that
\begin{align*}
\brs{\N \brs{A}^2}^2 \leq&\ \brs{\N A}^2 \brs{A}^2.
\end{align*}
Also, by the Cauchy-Schwarz inequality it follows that
\begin{align*}
\IP{Q,\tr_g (A \otimes \bar{A})} \leq \brs{Q} \brs{A}^2 \leq \brs{T}^2
\brs{A}^2.
\end{align*}
The corollary follows.
\end{proof}
\end{cor}

\section{Proofs of Theorems} \label{proofsec}

\begin{proof} [Proof of Theorem \ref{contramplethm}] Let $(M^{2n}, J)$ be a
compact complex manifold and let $E$ denote a contravariant globally generated
proper bundle.  We claim that, given a background Hermitian metric $h$, one has
\begin{align} \label{contra10}
 g_t \geq C^{-1} h,
\end{align}
for some uniform constant $C$ and any smooth existence time $t$.  First we
observe that since $M$ is compact, the space of holomorphic sections of $E$ is
finite dimensional, and we choose a basis $\{\gs_i\}$.  Declaring $\gs_i(x,t) =
\gs_i(x)$ one has that
\begin{align*}
 \dt \gs_i = \gD \gs_i.
\end{align*}
It follows directly
from the maximum principle applied to the result of Corollary \ref{pformevcor}
that
\begin{align*}
 \sup_{M \times \{t\}} \brs{\gs_i}_{g_t}^2 \leq \sup_{M \times \{0\}}
\brs{\gs_i}_{g_0}^2 \leq C .
\end{align*}
Since the $\gs_i$ form a finite spanning set at each point $p$, it follows that
the induced metric on $E$ is bounded above.  Since $E$ is proper this implies
that the metric on $T^*$, i.e. $g^{-1}$, is bounded above.  Thus the claim of
(\ref{contra10}) follows.  The statement of existence on $[0,\tau^*)$ follows
directly from Theorem \ref{uplowbndthm} and Theorem \ref{EKthm}.

Now we establish the statement of convergence.  Let $h$ denote a background
Hermitian metric for which $\rho(h) = 0$.  Combining Lemmas \ref{volumeformev}
and \ref{delalphaev} we obtain that
\begin{align*}
\left(\dt - \gD \right) \left[ \log \frac{\det g}{\det h} + \brs{\del \ga}^2
\right] \leq&\ 0.
\end{align*}
It follows from the maximum principle that
\begin{align*}
g_t \leq C h, \qquad \brs{\del \ga}^2 \leq C.
\end{align*}
It now follows directly from Theorem \ref{EKthm} and Corollary \ref{convergencecor} that the flow exists smoothly for all time and converges to a Calabi-Yau metric, as claimed.
\end{proof}

\begin{proof} [Proof of Theorem \ref{covamplethm}]  Let $(M^{2n}, J)$ be a
compact complex manifold with $c_1^{BC} = 0$, and let $E$ denote a covariant
globally generated weakly proper bundle.  We claim that, given a background
Hermitian metric $h$, one has
\begin{align} \label{cov10}
 g_t \leq C h,
\end{align}
for some uniform constant $C$ and any smooth existence time $t$.  As above, since $M$ is compact, the space of holomorphic sections of $E$ is
finite dimensional, and we choose a basis $\{\gs_i\}$.  Declaring $\gs_i(x,t) =
\gs_i(x)$ one has that
\begin{align*}
 \dt \gs_i = \gD \gs_i.
\end{align*}
Since $c_1^{BC} = 0$ and $[\del \gw_0] = 0$ we can choose $\eta$ and $\phi$ as
in Lemma \ref{delalphaev} so that
\begin{align*}
\left( \dt - \gD \right) \brs{\phi}^2 \leq&\ - \brs{T}^2.
\end{align*}
Now define
\begin{align*}
\Phi =&\ \log \brs{\gs_i}^2 + p \brs{\phi}^2
\end{align*}
It follows from Lemma \ref{delalphaev} and Corollary \ref{ptensorev} that
\begin{align*}
\left(\dt - \gD \right) \Phi \leq&\ 0.
\end{align*}
Note that we can still apply the maximum principle to $\Phi$ at maximum points
even though it approaches $-\infty$ at the vanishing locus of $\gs_i$.  It
follows that
\begin{align*}
\sup_{M \times \{t\}} \brs{\gs_i}^2_{g_t} \leq&\ \sup_{M \times \{0\}}
\brs{\gs_i}^2_{g_0} \leq C.
\end{align*}
Since the $\gs_i$ form a finite spanning set at each point $p$, it follows that
the induced metric on $E$ is bounded above.  Since $c_1^{BC} = 0$ we choose a
Hermitian metric $h$ such that $\rho(h) = 0$.  It follows from Lemma
\ref{volumeformev} that
\begin{align*}
\left( \dt - \gD \right) \log \frac{\det g}{\det h} =&\ \brs{T}^2 \geq 0.
\end{align*}
It follows from the maximum principle that
\begin{align*}
\inf_{M \times \{t\}} \frac{\det g_t}{\det h} \geq \inf_{M \times \{0\}}
\frac{\det g_0}{\det h} \geq C^{-1}.
\end{align*}
As we have established an upper bound for the induced metric on $E$ and a lower
bound on the volume form, since the bundle $E$ is weakly proper it now follows
that the metric $g_t$ is bounded above.  But again since the volume form is
bounded below it follows that $g_t$ is bounded below as well.  It follows
directly from Theorem \ref{EKthm} and \ref{convergencecor} that the flow exists
smoothly for all time and converges to a Calabi-Yau metric.
\end{proof}

\begin{proof} [Proof of Corollary \ref{GKcor}] Let $(M^{2n}, I, J_t, g_t)$ be
the solution to generalized K\"ahler-Ricci flow in the $I$-fixed gauge, as
explained in \S \ref{gksec}.  This means that $(M, I, g_t)$ is a solution to
pluriclosed flow.  In either case of the Corollary, using Theorem
\ref{contramplethm} or \ref{covamplethm} we obtain the long time existence and
exponential convergence of the flow to a Calabi-Yau manifold.  In particular,
the torsion is decaying exponentially to zero, and so the vector field defining
the diffeomorphisms $\phi_t$ such that $J_t = \phi_t^* J$ are converging
exponentially fast to a limiting diffeomorphism $\phi_{\infty}$, and hence $J_t$
is converging to a limit $J_{\infty}$.  The corollary follows.
\end{proof}

\bibliographystyle{hamsplain}

\begin{thebibliography}{10}

\bibitem{Akhiezer} D. Akhiezer, \emph{Lie groups actions in complex analysis},
Aspects of Mathematics, Vol. E 27

\bibitem{Bryant} R.L. Bryant, \emph{Rigidity and quasi-rigidity of extremal
cycles in Hermitian symmetric spaces}, Princeton University Press, 2010.

\bibitem{Gualtieri} Gualtieri, M. \emph{Generalized complex geometry}, Ann. of
Math.  Vol. 174 (2011), 75-123.

\bibitem{SBIPCF} J. Streets, \emph{Pluriclosed flow Born-Infeld geometry, and
rigidity results for generalized K\"ahler manifolds}, \textsf{arXiv:1502.02584},
to appear
Comm. PDE.

\bibitem{SPCFSTB} J. Streets, \emph{Pluriclosed flow on generalized K\"ahler
manifolds with split tangent bundle}, \textsf{arXiv:1405.0727}.

\bibitem{ST1} J. Streets, G. Tian,
\emph{Hermitian curvature
flow}, JEMS Vol. 13, no. 3 (2011), 601-634.

\bibitem{ST2} J. Streets, G. Tian, \emph{A
parabolic flow of pluriclosed
metrics}, Int. Math. Res. Notices (2010), Vol. 2010, 3101-3133.

\bibitem{ST3} J. Streets, G. Tian,  \emph{Regularity results for the
pluriclosed
flow}, Geom. \& Top. 17 (2013) 2389-2429.

\bibitem{STGK} J. Streets, G. Tian, \emph{Generalized K\"ahler geometry and
the
pluriclosed flow}, Nuc. Phys. B, Vol. 858, Issue 2, (2012) 366-376.

\end{thebibliography}

\end{document}